\documentclass[11pt]{amsart}

\usepackage{tikz}
\usetikzlibrary{positioning}
\usetikzlibrary{arrows}

\usepackage{amsmath,amsfonts}
\usepackage{mathrsfs,nccmath,amssymb,amsthm,mathtools}
\usepackage{mathabx}
\usepackage{enumitem}
\usepackage{appendix}
\usepackage[all]{xy}
\usepackage{color}
\usepackage{chngcntr}
\counterwithin{figure}{section}
\usepackage[normalem]{ulem}
\renewcommand{\le}{\varleq}
\renewcommand{\ge}{\vargeq}

\usepackage[fontsize=11pt]{scrextend}
\usepackage{stackengine}
\usepackage{hyperref}
\usepackage[normalem]{ulem}
\renewcommand{\le}{\leq}
\renewcommand{\ge}{\geq}

\newcommand{\alignb}{\[\begin{aligned} }
\newcommand{\alignn}{ \end{aligned}\]}

\newcommand\mL{L\kern-0.08cm\char39}
\newcommand{\myforall}{\text{ for all }}

\newcommand{\myand}{\text{ and }}

\newcommand{\seb}{\{\,}
\newcommand{\sen}{\,\}}

\newcommand{\clos}[1]{\mkern 1.5mu\overline{\mkern-1.5mu#1\mkern-1.5mu}\mkern 1.5mu}

\newcommand{\Ccal}{\mathcal{C}}

\newcommand{\Mcal}{\mathcal{M}}

\newcommand{\Pcal}{\mathcal{P}}

\newcommand{\Scal}{\mathcal{S}}

\newcommand{\kuu}{\emptyset}
\newcommand{\nekuu}{\neq \kuu}
\newcommand{\iskuu}{= \kuu}

\newcommand{\fai}{\varphi}

\newcommand{\barB}{\bar{B}}

\newcommand{\barU}{\bar{U}}

\newcommand{\pdirectional}{\raise0.05em\hbox{$+$}directional}
\newcommand{\pdirectionality}{\raise0.05em\hbox{$+$}directionality}
\newcommand{\pdirectionalitys}{\raise0.05em\hbox{$+$}directionality }
\newcommand{\pdirectionals}{\raise0.05em\hbox{$+$}directional }
\newcommand{\mdirectional}{\raise0.05em\hbox{$-$}directional}
\newcommand{\mdirectionality}{\raise0.05em\hbox{$-$}directionality}
\newcommand{\mdirectionalitys}{\raise0.05em\hbox{$-$}directionality }
\newcommand{\mdirectionals}{\raise0.05em\hbox{$-$}directional }

\newcommand{\Z}{\mathbb{Z}}

\newcommand{\R}{\mathbb{R}}

\newcommand{\bi}{\in \Z}

\newcommand{\bpi}{\ge 1}

\newcommand{\bni}{\ge 0}

\newcommand{\diam}{{\rm diam}}

\newcommand{\ep}{\varepsilon}

\newcommand{\centb}{\begin{center}}
\newcommand{\centn}{\end{center}}
\newcommand{\enumb}{\begin{enumerate}}
\newcommand{\enumn}{\end{enumerate}}
\newcommand{\itemb}{\begin{itemize}}
\newcommand{\itemn}{\end{itemize}}
\numberwithin{equation}{section}
\usepackage[left=30mm,right=30mm,top=25mm,bottom=25mm]{geometry}
\usepackage{cleveref}
\setlist[enumerate,1]{label=(\alph*),ref=(\alph*)}
\setlist[enumerate,2]{label=(\arabic*),ref=(\alph{enumi}-\arabic{enumii})}
\setlist[enumerate,3]{label=(\Alph*),ref=(\roman{enumi}-\alph{enumii}-\Alph*)}
\setlist[enumerate,4]{label=(\arabic*),ref=(\roman{enumi}-\alph{enumii}-\Alph{enumiii}-\arabic*)}
\newlist{enumrm}{enumerate}{1}
\setlist[enumrm,1]{label={\rm (\roman*)},ref=(\roman*)}
\newcommand{\enumrmb}{\begin{enumrm}}
\newcommand{\enumrmn}{\end{enumrm}}
\newlist{enumsec}{enumerate}{1}
\crefalias{enumseci}{enumi} 
\setlist[enumsec,1]{label={\rm (\thesection.\alph*)},ref=(\thesection.\alph*)}
\newcommand{\enumsecb}{\begin{enumsec}[resume]}
\newcommand{\enumsecn}{\end{enumsec}}
\newlist{deepenum}{enumerate}{1}
\setlist[deepenum,1]{label=($1$:\alph*),ref=(1:\alph*)}
\newlist{Lminusoneenum}{enumerate}{1}
\setlist[Lminusoneenum,1]{label=($-1$:\alph*),ref=($-1$:\alph*)}
\newlist{Ldeepenum}{enumerate}{1}
\setlist[Ldeepenum,1]{label=($2$:\alph*),ref=($2$:\alph*)}
\numberwithin{equation}{section}
\newtheorem{thm}[equation]{Theorem}%[section]
\newtheorem{lem}[equation]{Lemma}
\newtheorem{prop}[equation]{Proposition}
\newtheorem{cor}[equation]{Corollary}

\theoremstyle{definition}
\newtheorem{defn}[equation]{Definition}

\newtheorem{example}[equation]{Example}

\theoremstyle{remark}

\newtheorem{nota}[equation]{Notation}
\newtheorem{rem}[equation]{Remark}

\crefname{sec}{\S}{\S\S}
\crefname{mainthm}{Theorem}{Theorems}
\crefname{maincor}{Corollary}{Corollaries}
\crefname{thm}{Theorem}{Theorems}
\crefname{lem}{Lemma}{Lemmas}
\crefname{prop}{Proposition}{Propositions}
\crefname{cor}{Corollary}{Corollaries}
\crefname{defn}{Definition}{Definitions}
\crefname{conj}{Conjecture}{Conjectures}
\crefname{example}{Example}{Examples}
\crefname{nota}{Notation}{Notations}
\crefname{rem}{Remark}{Remarks}
\crefname{note}{Note}{Notes}
\crefname{case}{Case}{Cases}
\crefname{figure}{Figure}{Figures}
\crefname{section}{\S}{\S\S}
\crefname{enumi}{}{}
\crefname{enumii}{}{}
\crefname{equation}{}{}

\newcommand{\abs}[1]{|#1|}

\newcommand{\Vp}{V \setminus V_0}

\newcommand{\imply}{$\Rightarrow$}

\begin{document}

\title[Refinement of Bratteli--Vershik models]{Refinement of Bratteli--Vershik models}

\author{Takashi Shimomura}

\address{Nagoya University of Economics, Uchikubo 61-1, Inuyama 484-8504, Japan}
\curraddr{}
\email{tkshimo@nagoya-ku.ac.jp}
\thanks{}


\date{\today}

%\dedicatory{}

%\commby{}

\begin{abstract}
In the zero-dimensional systems, the Bratteli--Vershik models can be built
upon certain closed sets that are called `quasi-sections' in this article.
There exists
 a bijective correspondence between
the topological conjugacy
 classes of triples of zero-dimensional systems and quasi-sections
 and the topological conjugacy classes of Bratteli--Vershik models.
Therefore, we can get refined Bratteli--Vershik models if we get
 certain refined quasi-sections.
The basic sets are such refined quasi-sections that bring `closing property'
 on the corresponding Bratteli--Vershik models.
We show a direct proof on the existence of basic sets.
Thorough investigations on quasi-sections and basic sets
 are done.
Furthermore, it would be convenient
 for the Bratteli--Vershik models to concern minimal sets.
To this point, we show the existence of the Bratteli--Vershik models
 whose minimal sets are properly ordered.
On the other hand, we can get certain refinements with respect to
 the Bratteli--Vershikizability condition
 or the decisiveness.
\end{abstract}

\maketitle

\section{Introduction}\label{sec:introduction}
In this paper, we represent $(X,f)$ as a \textit{zero-dimensional system}
 if $X$ is a compact metrizable zero-dimensional space and 
 $f : X \to X$ is a homeomorphism
 (i.e., we consider only invertible zero-dimensional systems). 
For each $x \in X$, let $O(x) := \seb f^n(x) \mid n \in \Z \sen$,
 i.e., $O(x)$ is the orbit of $x$.
We say that a closed and open (clopen) set $A$ is a \textit{complete section}
 if $A \cap O(x) \nekuu$ for all $x \in X$.
It is easy to see that a clopen set $A$ is a complete section if and only
 if every positive orbit $x, f(x), f^2(x), \cdots$ enters $A$.
A closed set $B \subseteq X$ is called a \textit{quasi-section} if
 there exists a sequence $A_0 \supseteq A_1 \supseteq A_2 \supseteq \cdots$
 of complete sections such that $B = \bigcap_{n =0}^{\infty}A_n$.
The notion of quasi-section (with a different terminology)
 has been introduced by
 Poon \cite[\S 1]{Poon1990AFSubAlgebrasOfCertainCrossedProducts}
 and has been used to study the $C^*$-crossed product.
What is meaningful to mention is that Poon used quasi-sections to produce
 some AF-algebras.
For every $x \in X$, the $\omega$-limit set (see \cref{defn:omegalimitset})
 of $x$ is denoted by $\omega(x)$.
It is quite easy to see that $B$ is a quasi-section if and only if
 the intersection $B \cap \omega(x)$ is not empty for all $x \in X$,
 (see \cref{lem:qs} for details).
We call
 a quasi-section $B \subseteq X$ is a \textit{basic set}
 if every orbit passes through $B$ at most once, i.e.,
 a quasi-section $B$ is a basic set if and only if
 $\abs{O(x) \cap B} \le 1$ for all $x \in X$.
The existence of basic sets for an arbitrary zero-dimensional systems
 have been given in 
 \cite[Theorem 3.34]{Shimomura_2020AiMBratteliVershikModelsAndGraphCoveringModels}.
For Bratteli--Vershik models on an arbitrary zero-dimensional systems $(X,f)$,
 let us consider the topological conjugacy of a triple $(X,f,B)$
 with a quasi-section $B$ corresponding to the set of minimal paths 
 of the ordered Bratteli diagram.
The Bratteli--Vershik models are well known to be a key 
to understanding the $C^*$-crossed product.
In \cite{HERMAN_1992OrdBratteliDiagDimGroupTopDyn}, Herman, Putnam and Skau,
 in these context,
 showed the bijective correspondence
 between equivalence classes of essentially simple ordered Bratteli diagrams
 and pointed topological conjugacy classes of essentially minimal systems,
 i.e., they considered the class of triples $(X,f,x_0)$ of essentially 
 minimal system $(X,f)$ and a point $x_0$ from the minimal set.
In their work, this bijective correspondence extends to  
 isomorphism classes of certain $C^*$-algebraic objects.
They used the pointed topological conjugacy because the set of minimal
 paths of an essentially simple ordered Bratteli diagram is a single point.
In \cite{Medynets_2006CantorAperSysBratDiag}, for Cantor systems that have no periodic points,
 Medynets defined
 basic sets, proved their existence, and constructed the
 corresponding Bratteli--Vershik models, getting out of pointed topological
 conjugacy.
He also studied the $K$-theory of zero-dimensional systems that have no periodic points.
In this paper, our work does not extend to $C^*$-algebraic objects.
However, the restriction of aperiodicity is now removed, i.e., we treat all zero-dimensional systems.
In the earlier versions of this manuscript,
 I established the bijective correspondence
 in the following \cref{thm:bijectivetriplesBVmodels}.
After these,
 Golestani, Hosseini, and Oghli,
 in \cite[\S 4]{GHO_2023TopologicalFactoringOfZeroDimensionalDynamicalSystems},
 have shown that the basis of \cref{thm:bijectivetriplesBVmodels}
 had already been established by 
 Putnam in the proof of 
 \cite[Lemma 3.1]{Putnam89TheCstarAlgebrasAssociatedWithMinimalHomeomorphismsOfTheCantorSet} and reestablished the proof in
 \cite[Proposition 4.4]{GHO_2023TopologicalFactoringOfZeroDimensionalDynamicalSystems}.
\begin{thm}[Putnam {\cite[Lemma 3.1]{Putnam89TheCstarAlgebrasAssociatedWithMinimalHomeomorphismsOfTheCantorSet}}]\label{thm:bijectivetriplesBVmodels}
There exists a conjugacy-preserving bijective correspondence between topological conjugacy 
 classes of triples of zero-dimensional systems and quasi-sections
 and topological conjugacy classes of Bratteli--Vershik models.
\end{thm}
In the above study, Putnam described the proof
 in the context of
 the Cantor minimal systems.
Therefore, I think it is still worthwhile to describe the proofs
 in the context of general zero-dimensional systems.
I have given a thorough proof in this paper (see the proof of \cref{thm:from-quasi-section-to-KR}).

In the above clarification of the bijective correspondence,
 we can get the possibility of getting some refined Bratteli--Vershik models
 by finding some refined quasi-sections.
In the context of the Bratteli--Vershik models, the notion of basic sets
 corresponds
 to the notion of `closing property' (see \cref{defn:keeping-constant})
 that was introduced
 in \cite{Shimomura_2020AiMBratteliVershikModelsAndGraphCoveringModels}.
Thus, we get a refined bijective correspondence as follows:
\begin{thm}\label{thm:bijectivebasicsetsBVmodelswithclosingproperty}
There exists a conjugacy-preserving bijective correspondence between topological conjugacy 
 classes of triples of zero-dimensional systems and 
 basic sets
 and topological conjugacy classes of 
 Bratteli--Vershik models
 with the closing property.
Furthermore, these classes are not empty
 for an arbitrary zero-dimensional systems.
\end{thm}

Much more refined bijective correspondence is meaningful by 
 finding some refined basic sets for all zero-dimensional systems.
Bezuglyi, Niu and Sun in \cite{BezuglyiNiuSun_2021CstarAlgOfaCantorSysWithFinitelyManyMinStructKTtheAndIndexMap}
 studied $C^*$-algebra of zero-dimensional systems
 with finitely many minimal sets.
If the number of minimal sets is $k$, then they constructed 
Bratteli--Vershik models considering the minimal sets.
They defined $k$-simple Bratteli diagrams for a positive integer $k$
 (the number of minimal sets).
For an arbitrary zero-dimensional system, we say that
 a Bratteli--Vershik model is `quasi-simple' if
 the restriction to each minimal set is properly ordered
 (see \cref{defn:simple,defn:oBd-properly-ordered,defn:BV-quasi-simple}).
To specify, for a triple $(X,f,B)$, we say that
 a quasi-section $B$ is \textit{quasi-simple} if $\abs{B \cap M}=1$
 for each minimal set $M$ of $(X,f)$.
We show the existence of quasi-simple basic sets in \cref{thm:inf_f}.
Thus, we get the following:
\begin{thm}\label{thm:bijectiveqsbasicsetsBVmodelswithssclosingproperty}
There exists a conjugacy-preserving bijective correspondence between topological conjugacy 
 classes of triples of zero-dimensional systems and 
 quasi-simple quasi-sections
 and topological conjugacy classes of 
 quasi-simple Bratteli--Vershik models.
Furthermore, these classes are not empty
 for an arbitrary zero-dimensional systems.
\end{thm}

Another way of refining the Bratteli--Vershik models was
 found by Downarowicz and Karpel in
 \cite{DownarowiczKarpel_2018DynamicsInDimensionZeroASurvey,DownarowiczKarpel_2019DecisiveBratteliVershikmodels}.
This enables us to find some
 bijective correspondences between triples and
 ordered Bratteli diagrams.
The \cref{thm:bijectivetriplesBVmodels} does not imply some bijective
 correspondences between triples and ordered Bratteli diagrams.
In \cite{BezuglyiYassawi2017OrdersThatYieldHomeoOnBratteliDiagrams,BezuglyiKwiatkowskiYassawi_2014PerfectOrderingsOnFiniteRankBratteliDiagrams},
 Bezuglyi, Yassawi, and Kwiatkowski considered the condition on
 Bratteli diagrams of having continuous Vershik maps.
In general, an ordered Bratteli diagram may not have a unique Vershik map
 even if it has a continuous Vershik map.
Therefore, we need the work of Downarowicz and Karpel
 \cite{DownarowiczKarpel_2019DecisiveBratteliVershikmodels}.
They presented the notion of
 decisiveness on ordered Bratteli diagrams,
as a result of which the ordered Bratteli diagrams can have 
 unique Vershik maps (see \cref{defn:decisive}).
They also introduced the notion of the Bratteli--Vershikizability
 condition
 for zero-dimensional systems, i.e.,
 a zero-dimensional system is called Bratteli--Vershikizable
 if it is conjugate to a Bratteli--Vershik model constructed from 
 a decisive ordered Bratteli diagram (cf. \cref{defn:Bratteli--Vershikizable}).
In their main theorem
 \cite[Theorem 3.1]{DownarowiczKarpel_2019DecisiveBratteliVershikmodels},
 they showed that a zero-dimensional system
 is Bratteli--Vershikizable
 if and only if either the set of aperiodic points is
 dense, or its closure misses one periodic orbit.
Here, as an application of
 the bijective correspondence \cref{thm:bijectivetriplesBVmodels},
 we give a new way of proving their main result
 (see the proof of \cref{thm:Bratteli-Vershikizability-new}).
Under the Bratteli--Vershikizability condition,
 we can also obtain a refined Bratteli--Vershik models
 by selecting quasi-simple decisive basic sets
 (cf. \cref{thm:quasi-simple-decisive-and-closing}).

Owing to their work, we can now delve into bijective correspondences
 between triples and ordered Bratteli diagrams.
Their main theorem
 \cite[Theorem 3.1]{DownarowiczKarpel_2019DecisiveBratteliVershikmodels}
 enables us to concentrate on the systems that have dense aperiodic points.
In this paper, a zero-dimensional system is called 
 \textit{densely aperiodic} if it has dense aperiodic points.
In \cite{Poon1990AFSubAlgebrasOfCertainCrossedProducts},
 Poon studied such zero-dimensional systems
 (cf. \cite[Theorem 4.3]{Poon1990AFSubAlgebrasOfCertainCrossedProducts}).
An ordered Bratteli diagram is called \textit{continuously decisive} if
 it is decisive and
 the set of maximal paths contains empty interior.
We also provide the following definition:
 a quasi-section $B$ is \textit{continuously decisive}
 if $\textrm{int}B \iskuu$, and a triple $(X,f,B)$ is
 \textit{continuously decisive} if $B$ is continuously decisive.
Thus, we can obtain a desired bijective correspondence.
\begin{thm}\label{thm:bijective-continuously-decisive}
There exists a bijective correspondence between the equivalence classes of
 continuously decisive ordered Bratteli diagrams
 and the topological conjugacy
 classes of continuously decisive triples of
 zero-dimensional systems with quasi-sections.
\end{thm}
In particular, it is easy to obtain distinctive bijections
 by choosing subclasses of quasi-sections
 (see \cref{cor:bijective-continuously-decisive-closing,cor:bijective-continuously-decisive-quasi-simple}).

In \cref{sec:preliminaries}, we introduce certain essential definitions
 and notations.
Concerning the notions of quasi-section, basic set, and its minimality;
 initially, we analyze certain basics of 
 zero-dimensional analysis to prepare the base for later sections.
In \cref{sec:quasi-section-and-KR-refinement},
 we introduce the notion of Kakutani--Rokhlin refinements
 (abbrev. K--R refinements) and show
 \cref{thm:from-quasi-section-to-KR}.
In \cref{sec:main-basic-results}, we introduce the Bratteli--Vershik models.
Through \cref{sec:quasi-section-and-KR-refinement,sec:main-basic-results},
 we link previously defined notions and \cref{thm:from-quasi-section-to-KR}.
In \cref{subsec:proofofbasicresults}, we show a proof of
  \cref{thm:bijectivetriplesBVmodels,thm:bijectivebasicsetsBVmodelswithclosingproperty,thm:bijectiveqsbasicsetsBVmodelswithssclosingproperty}.
In \cref{sec:decisiveness}, relations between quasi-sections and 
 decisiveness are considered and the bijective correspondences
 in the level of ordered Bratteli diagrams are proved.
Finally, in \cref{sec:applic}, we present some applications.
In particular, in \cref{cor:tt-closing-decisive},
 we prove that if a zero-dimensional system is topologically transitive, then closing property
on a Bratteli--Vershik model  implies its decisiveness.

\section{Preliminaries and Notation}\label{sec:preliminaries}
Let $\Z$ be the set of all integers.
For integers $a < b$, the intervals are denoted by
 $[a,b] := \seb a, a+1, \dotsc,b \sen$, and so on.
Let $(X,f)$ be a zero-dimensional system and $d$ be a metric on $X$.
For a subset $A \subseteq X$,
 $\diam(A) := \sup \seb d(x,y) \mid x,y \in A \sen$.
We denote
 the interior of $A$ in $X$ by $\textrm{int}A$.
The set of aperiodic points is denoted as $A_f$, i.e.,
 $A_f = \seb x \mid f^n(x) \ne x \textrm{ for all } n \ne 0 \sen$.
In addition, for a subset $A \subseteq X$, we denote the orbit 
 as $O(A) := \bigcup_{n \bi}f^n(A)$.
If $f(Y) = Y$ for a non-empty closed set $Y \subseteq X$, 
then $(Y,f|_Y)$ is called a \textit{subsystem} of $(X,f)$.
A non-empty closed set $M \subseteq X$ is called a \textit{minimal set} 
 if $f(M) = M$ and every orbit in $M$ is dense in $M$.
We denote $\Mcal_f := \seb M \mid M \text{ is a minimal set.} \sen$.

\begin{defn}\label{defn:omegalimitset}
Let $(X,f)$ be a zero-dimensional system.

The positive orbit of $x$ is denoted as $O_+(x)$, i.e.,
 $O_+(x) = \seb f^n(x) \mid n \ge 0 \sen$.

The negative orbit of $x$ is denoted as $O_-(x)$, i.e.,
 $O_-(x) = \seb f^n(x) \mid n \le 0 \sen$.

For each point $x \in X$, the \textit{$\omega$-limit set} $\omega(x)$
 is defined
 as $\omega(x) := \bigcap_{n > 0}\clos{O_+(f^n(x))}$.

For each point $x \in X$, the \textit{$\alpha$-limit set} $\alpha(x)$
 is defined
 as $\alpha(x) := \bigcap_{n < 0}\clos{O_-(f^n(x))}$.
\end{defn}

A point $x \in X$ is said to be \textit{wandering} if there exists
 an open set $U \ni x$ such that $f^n(U) \cap U \iskuu$ for all $n \ne 0$.
Thus, we define  $\Omega_f := \seb x \mid x \text{ is not wandering. } \sen$.
The set $\Omega_f$ is said to be the \textit{non-wandering set} of $(X,f)$.
\begin{nota}\label{nota:equivalence-for-triple}
Let $(X,f,B)$ and $(Y,g,B')$ be two triples such that
 both $(X,f)$ and $(Y,g)$ are zero-dimensional systems, and
 both $B \subseteq X$ and $B' \subseteq Y$ are quasi-sections.
We say that $(X,f,B)$ and $(Y,g,B')$ are \textit{topologically conjugate}
 if there exists a homeomorphism $\fai :X \to Y$ such that
 $\fai \circ f = g \circ \fai$ and $B' = \fai(B)$.
\end{nota}

\begin{defn}
A zero-dimensional system $(X,f)$ is \textit{topologically transitive} if,
 for every pair of nonempty open sets $U,V \subseteq X$,
 there exists a positive integer $n$ such that
 $f^n(U) \cap V \nekuu$.
\end{defn}
It is self-evident that if $(X,f)$ is topologically transitive,
 then $\Omega(f) = X$.
Following Medynets \cite{Medynets_2006CantorAperSysBratDiag},
 we apply the following definition. 
For a zero-dimensional system $(X,f)$,
 let $h$ be a positive integer and $U \subseteq X$ be
 a non-empty closed and open set.
If all $f^i(U)$ $(0 \le i < h)$ are mutually disjoint,
 then $\barU := \seb f^i(U) \mid 0 \le i < h \sen$ is called a \textit{tower}
 with \textit{base} $U$ and \textit{height} $h$.
In this context, we say that each $f^i(U)$ ($0 \le i < h$)
 is a \textit{floor}.
We say that \textit{the diameter of the tower is $\epsilon$} if
 $\epsilon = \max_{i \in [0,h]} \diam(f^i(U))$.
We include $i = h$ when we calculate the diameters of towers.
We denote that $\bigcup \barU := \bigcup_{0 \le i < h}f^i(U)$.
The floor $f^{h-1}(U)$ is called the \textit{top floor}.
The notion of the tower has played a central role in
 the study of the Bratteli--Vershik models,
 a tower corresponds to a vertex of the related Bratteli diagram
 (see \cref{prop:BV-to-KR,prop:KR-to-ordered-Bratteli-diagram}).

\subsection{Basics of quasi-sections and basic sets}\label{subsec:basicfactquasisec_basicsets}

If we shall have accomplished the proof of \cref{thm:bijectivetriplesBVmodels},
 then many properties of Bratteli--Vershik models can be transferred
 to the quality of quasi-sections.
Therefore, it is worthwhile to conduct 
 a basic topological study on quasi-sections as well as basic sets.

Firstly, we note that a clopen set $U$ is a complete section
 if and only if $O(U) =X$.
In the study of
 the $C^*$-algebras of zero-dimensional systems, Poon \cite{Poon1990AFSubAlgebrasOfCertainCrossedProducts} considered
 closed sets such that every clopen neighborhood $U$ satisfies $O(U) = X$.
We shall check it out in the next lemma that
 such closed sets are exactly quasi-sections.

\begin{lem}\label{lem:qs}
Let $(X,f)$ be a zero-dimensional system and 
 $A \subseteq X$ be a closed set.
Then, the following statements are equivalent:
\enumb
\item\label{qs:qs} $A$ is a quasi-section,
\item\label{qs:compsec} every clopen set $U \supseteq A$ is a complete section,
\item\label{qs:orbit} $\clos{O(x)} \cap A \nekuu$ for every $x \in X$,
\item\label{qs:alpha} $\alpha(x) \cap A \nekuu$ for every $x \in X$, and
\item\label{qs:omega} $\omega(x) \cap A \nekuu$ for every $x \in X$, and
\item\label{qs:minimal} $A \cap M \nekuu$ for every $M \in \Mcal_f$.
\enumn
\end{lem}
\begin{proof}
To show \cref{qs:qs} \imply \cref{qs:compsec},
 let $A$ be a quasi-section; $U \supseteq A$, a clopen set;
 and $A_0 \supseteq A_1 \supseteq A_2 \supseteq \cdots$,
 a sequence of complete sections such that $\bigcap_{n=0}^{\infty}A_n = A$.
By the compactness of $X$, there exists an $n$ such that
 $U \supseteq A_n$.
Now, it is evident that $U$ is a complete section.
We show that \cref{qs:compsec} \imply \cref{qs:orbit}.
Let $A$ be a closed set.
Suppose that $O(U) = X$ for every clopen set $U \supseteq A$.
Let $x \in X$.
Then, it follows that $O(x)$ enters every open neighborhood of $A$.
Thus, it is clear that $\clos{O(x)} \cap A \nekuu$, as desired.
To show \cref{qs:orbit} \imply \cref{qs:alpha}, let
 $x \in X$.
Let $y \in \alpha(x)$.
Then, by \cref{qs:orbit}, it follows that $\clos{O(y)} \cap A \nekuu$.
Thus, $\alpha(x) \cap A \supseteq \clos{O(y)} \cap A \nekuu$, as desired.
In the same way,
 it is now evident that \cref{qs:orbit} \imply \cref{qs:omega}.
To show \cref{qs:alpha} \imply \cref{qs:minimal},
 let $M \in \Mcal_f$.
We have to show that $M \cap A \nekuu$.
However, this is evident because every $\alpha$-limit set of the point
 in $M$ has a non-empty intersection with $A$, by \cref{qs:alpha}.
In the same way, it is now evident
 that \cref{qs:omega} \imply \cref{qs:minimal}.
Finally, we show that \cref{qs:minimal} \imply \cref{qs:qs}.
It is easy to construct a sequence of clopen sets
 $A_0 \supseteq A_1 \supseteq A_2 \supseteq \cdots$ such that 
 $\bigcap_{n=0}^{\infty}A_n = A$.
For each $n \ge 0$ and each $x \in X$,
 it is evident that $A_n \cap O(x) \nekuu$.
This implies that each $A_n$ is a complete section.
\end{proof}

The minimal (with respect to inclusions) quasi-sections were 
 used in
 \cite[\S 4]{Poon1990AFSubAlgebrasOfCertainCrossedProducts}.

\begin{lem}\label{lem:minimalqs}
Every quasi-section contains a minimal quasi-section.
\end{lem}
\begin{proof}
The proof follows directly from \cref{qs:minimal} of \cref{lem:qs}
 in combination with
 the assumption that $X$ is a compact metrizable space.
\end{proof}

\begin{prop}\label{prop:minimalqscontdic}
Suppose that $(X,f)$ is densely aperiodic.
Then, every minimal quasi-section is continuously decisive.
In particular, every densely aperiodic system has a continuously decisive
 quasi-section.
\end{prop}
\begin{proof}
Let $B$ be a minimal quasi-section.
Suppose that $U := \textup{int}B \nekuu$.
First, suppose that there exists an $x \in U$ and an $n \ne 0$ such that
 $f^n(x) = x$.
Then, there exists an aperiodic point $y \in U$ that is close to $x$ such that
 $f^n(y) \in U$ because $(X,f)$ is densely aperiodic by the assumption.
Henceforth, there exists an open set $V$ $(y \in V \subseteq U)$ such that
  $V \cap f^n(V) \iskuu$ and $f^n(V) \subseteq U$.
In this case, if we define $B' := B \setminus V$, then 
 $B'$ is a quasi-section because every orbit that passes through $V$ also
 passes through $f^n(V) \subseteq B'$.
This contradicts the minimality of $B$.
Next, suppose that there does not exist any periodic point in $U$.
We shall show that $f^n(U) \cap U \iskuu$ for all $n \ne 0$.
Suppose that there exists an $n \ne 0$ with $f^n(U) \cap U \nekuu$.
Then, there exists an $x \in U$ with $f^n(x) \in U$.
It follows that $x \ne f^n(x)$ because no periodic point exists in $U$.
Therefore, there exists an open set $V$ ($x \in V \subset U$) with
 $V \cap f^n(V) \iskuu$ and $f^n(V) \subseteq U$.
Again, we have a contradiction as in the first case.
Therefore, we obtain $f^n(U) \cap U \iskuu$ for all $n \ne 0$.
This shows that $U \cap M \iskuu$ for all $M \in \Mcal_f$.
Thus, $B' := B \setminus U$ is a quasi-section.
This contradicts the minimality of $B$.
This completes the proof of the first half of the statement.
The second half of the statement is shown from the first half and \cref{lem:minimalqs}.
\end{proof}

We shall show the examples of some primitive quasi-sections.
\begin{example}\label{example:quasi-section-continuous-map}
Let $(X,f)$ be a Cantor minimal system, $C$ be the Cantor set,
 and $\fai : C \to X$ be a continuous map.
We define a homeomorphism $f \times id : X\times C \to X\times C$
 as $(f \times id)(x,y) = (f(x),y)$ for all $(x,y) \in X \times C$.
Then, it follows that
 $\Mcal_{f \times id} = \seb X \times \seb c \sen \mid c \in C \sen$ and that
 the graph $A := \seb (\fai(y),y) \mid y \in C\sen$ is closed.
Because $A \cap M \nekuu$ for all $M \in \Mcal_{f \times id}$,
 by \cref{lem:qs},
 $A$ is a quasi-section.
We also note that these are minimal quasi-sections.
\end{example}
\begin{example}\label{example:quasi-section-not-basic-set}
Let $(X,f)$ be a minimal set and $C$ be the Cantor set in the interval
 $[0,1] \subset \R$.
We consider a zero-dimensional system $f \times id$ as above.
Then, it follows that
 $\Mcal = \seb X \times \seb c \sen \mid c \in C \sen$.
Take an $a \in C$ such that neither $C \cap [0,a)$ nor $C \cap (a,1]$
 is closed in $C$.
Take $x,y \in X$ such that $x \ne y$.
We define that
 $A := (\seb x \sen \times (C\cap [0,a]))
 \cup (\seb y \sen \times (C\cap [a,1]))$.
It then follows that $A$ is closed.
Because $A \cap M \nekuu$ for all $M \in \Mcal$,
$A$ is a quasi-section.
We note that $A$ is also minimal.
If $y \in O(x)$, then $A$ is not a basic set.
Therefore, a minimal quasi-section may not be a basic set.
On the other hand, if $y \notin O(x)$, then $A$ is a minimal basic set.
In particular, there exists a minimal basic set $B$ that
 satisfies $\abs{B \cap M} \ge 2$ for some $M \in \Mcal$.
\end{example}

\subsubsection{Basic sets}
Medynets \cite{Medynets_2006CantorAperSysBratDiag} defined basic sets 
 for aperiodic zero-dimensional systems.
We have shown that there exists a basic set in every zero-dimensional system
(cf. \cite[Theorem 1.1 and Theorem 4.19]{Shimomura_2020AiMBratteliVershikModelsAndGraphCoveringModels}).
We shall show a direct proof on the existence of certainly refined basic sets
 later in this section
 (cf. \cref{thm:inf_f}).

\begin{lem}\label{lem:intBwander}
Let $(X,f)$ be a densely aperiodic zero-dimensional system.
Suppose that $B$ is a basic set.
Then, for every $n \ne 0$, $f^n(\textup{int}B) \cap \textup{int}B \iskuu$.
In particular, every point $x \in \textup{int}B$ is wandering.
\end{lem}
\begin{proof}
If $\textup{int}B \iskuu$, then the statement is self-evident.
Suppose that $U := \textrm{int} B \nekuu$.
Suppose, on the contrary, that there exists an $n \ne 0$ such that
 $f^n(U) \cap U \nekuu$.
By the assumption that $\clos{A_f} = X$, it follows that
 there exists an $x \in f^n(U) \cap U \cap A_f$.
Then, it follows that $\abs{O(x) \cap B} \ge 2$, which is a contradiction.
\end{proof}

In general, for a quasi-section $B$ in a densely aperiodic zero-dimensional
 system, the set $B\setminus \textup{int}B$
 may not be a quasi-section.
However, this is valid for any basic set.
\begin{prop}
Let $(X,f)$ be a zero-dimensional system.
Suppose that $(X,f)$ is densely aperiodic and $B$ is a basic set.
Then, $B' := B \setminus \textup{int}B$ is a continuously decisive basic set.
\end{prop}
\begin{proof}
It is clear that $B'$ is a closed set
 and $\abs{B' \cap O(x)} \le 1$ for all $x \in X$.
It is also clear that $\textup{int}B' \iskuu$.
We need to show that $B' \cap M \nekuu$ for every $M \in \Mcal_f$.
Fix an $M \in \Mcal_f$.
It then follows that $B \cap M \nekuu$; however, it also follows that
 $M \cap \textup{int}B \iskuu$ owing to \cref{lem:intBwander}.
Thus, we obtain $B' \cap M \nekuu$, as desired.
\end{proof}

\begin{thm}\label{thm:ap-non-wandering}
Let $(X,f)$ be a zero-dimensional system.
Suppose that $(X,f)$ is densely aperiodic
 and $\Omega_f = X$.
Then, every basic set is continuously decisive.
\end{thm}
\begin{proof}
Let $B$ be a basic set of $(X,f)$.
By \cref{lem:intBwander}, it follows that every point in $\textrm{int}B$ is 
 wandering.
Thus, by the assumption that $\Omega_f = X$, it follows that 
 $\textrm{int} B \iskuu$.
\end{proof}

\subsubsection{Quasi-simple basic set}
In this section, we present a direct proof of the existence of
 quasi-simple basic sets in every zero-dimensional system.
To do this, we embed $X$ into the real line $\R$, i.e.,
 $X \subset \R$. In particular,
 $X$ is linearly ordered, and the order topology coincides 
 with the original topology in $X$.

\begin{nota}
We use the notations 
$\inf_f(x) := \inf \seb y \mid y \in O(x) \sen$ and
 $\inf_f := \seb \inf_f(x) \mid x \in X \sen$.
\end{nota}

By the notation, for every $x \in X$, it follows that
 $\inf_f(x) \le x$ and $\inf_f(x) \in \overline{O(x)}$.
We obtain the following lemmas.
\begin{lem}\label{lem:basicinf}
Let $x \in \inf_f$.
 Then, it follows that $\inf_f(x) = x$.
\end{lem}
\begin{proof}
Let $x \in \inf_f$.
Then, there exists a $y \in X$ with $x = \inf_f(y)$.
Evidently, it follows that $\inf_f(\inf_f(y)) \le \inf_f(y)$.
We need to show $\inf_f(\inf_f(y)) = \inf_f(y)$.
Suppose, on the contrary, that $\inf_f(\inf_f(y)) < \inf_f(y)$.
Then, there exists an $n \bi$ such that
 $f^n(\inf_f(y)) < \inf_f(y)$.
If one chooses an $m \bi$ such that
 $f^m(y)$ is sufficiently close to $\inf_f(y)$,
 then one obtains $f^{n+m}(y) = f^n(f^m(y)) < \inf_f(y)$.
This contradicts the definition of $\inf_f(y)$.
\end{proof}

\begin{lem}\label{lem:inforbitinf}
It follows that $\inf_f$ is a closed set.
\end{lem}
\begin{proof}
Let $x_n$ ($n = 1,2,\dots$) be a sequence of points of $\inf_f$.
Suppose that $x_n \to x$ for some $x \in X$.
Suppose that $\inf_f(x) < x$.
Then, there exists an $i \bi$ such that $f^i(x) < x$.
Take an $\ep > 0$ such that $f^i(x) + \ep < x - \ep$.
It follows that $f^i(x_n) < f^i(x) + \ep$ for every sufficiently large $n$.
On the other hand, $x - \ep < x_n$ for every sufficiently large $n$.
Therefore, we obtain $f^i(x_n) < x_n$ for a sufficiently large $n$.
This contradicts \cref{lem:basicinf}.
\end{proof}

\begin{thm}\label{thm:inf_f}
Suppose that $(X,f)$ is a zero-dimensional system.
Then, the set
 $\inf_f$
 is a quasi-simple basic set.
In particular, every zero-dimensional system has a quasi-simple basic set.
\end{thm}

\begin{proof}
By \cref{lem:inforbitinf}, it follows that $\inf_f$ is closed.
Next, we show that $\inf_f$ is a quasi-section.
Let $M \in \Mcal_f$.
We need to show that $\inf_f \cap M \nekuu$.
Let $x \in M$.
Then, it follows that $\inf_f(x) \in \inf_f \cap M$, as desired.
To show that $\inf_f$ is a basic set,
 we need to show that $\abs{O(x) \cap \inf_f} \le 1$ for every
 $x \in X$.
Take $x,y \in \inf_f$ such that $O(x) \ni y$.
Then, by \cref{lem:basicinf}, 
 we obtain $\inf_f(x) = x$ and $\inf_f(y) = y$.
Therefore, we obtain
 $x = \inf_f(x) = \inf(O(x)) = \inf(O(y)) = \inf_f(y) = y$,
 as desired.
Finally, we have to show that $\abs{\inf_f \cap M} = 1$.
Let $M \in \Mcal_f$ and $x,y \in \inf_f \cap M$.
Then, it follows that
 $x = \inf_f(x) = \inf(\clos{O(x)})
 = \inf(M) = \inf(\clos{O(y)}) = \inf_f(y) = y$, as desired.
\end{proof}

\begin{rem}\label{rem:inf_funiqueminimal}
For each $M \in \Mcal$, let $x_M$ be a unique point in $\inf_f \cap M$.
Let $B' := \overline{\seb x_M : M \in \Mcal \sen}$.
Then, $B'$ is the unique minimal basic set in $\inf_f$.
In particular, $\inf_f$ has the unique minimal basic set.
\end{rem}

An $\inf_f$ basic set need not be minimal, as the following example shows.

\begin{example}\label{example:non-minimal-extream}
We show that there exists a zero-dimensional system $(X,f)$
 such that $\inf_f$ is not minimal.
Let us arbitrarily fix an embedding of $X$ into $\R$.
Then, $\inf_f$ is determined.
Suppose that $(X,f)$ contains two fixed points $p_1,p_2$
 and $\seb \seb p_1 \sen, \seb p_2 \sen \sen = \Mcal_f$.
In particular, it follows that $B_0 := \seb p_1, p_2\sen$ is
 the only minimal basic set.
Let us assume that there exist sequences
 $x_{i,n} \to p_i$ ($i = 1,2$ $n = 1,2,\dotsc$)
 of points of $X$
 such that, for all $n = 1,2, \dotsc$, it follows that
 $\alpha(x_{1,n}) = \omega(x_{1,n}) = p_2$ and
 $\alpha(x_{2,n}) = \omega(x_{2,n}) = p_1$.
It is self-evident that such a system exists.
In this system, one of the $p_i$s is less than the other.
Without loss of generality, let us assume that $p_2 < p_1$.
For all sufficiently large $n$, we obtain that $x_{2,n} < p_1$.
Thus, $\inf_f(x_{2,n}) < p_1$ for such $n$s.
For the orbit $O(x_{2,n})$,
 there exists a sole accumulation point $p_1$.
It follows that $\inf_f(x_{2,n}) \in O(x_{2,n})$ for such $n$s.
Thus, we obtain $\inf_f(x_{2,n}) \in \inf_f \cap O(x_{2,n})$ and
 $\inf_f(x_{2,n}) \notin B_0$.
This shows that $\inf_f$ is not minimal, as desired.
\end{example}

Summarizing the argument above, we obtain the following. 

\begin{thm}\label{thm:qs-main}
Let $(X,f)$ be a zero-dimensional system.
Suppose that $(X,f)$ is densely aperiodic.
Then, there exists a quasi-simple continuously decisive basic set $B$.
\end{thm}
\begin{proof}
By \cref{thm:inf_f}, there exists a quasi-simple basic set.
By \cref{lem:minimalqs}, there exists a minimal quasi-section
 $B' \subseteq B$
 that is also a basic set.
By \cref{prop:minimalqscontdic},
 such a $B'$ is continuously decisive.
\end{proof}

\section{Quasi-section and K--R refinement}\label{sec:quasi-section-and-KR-refinement}
To describe the proof of \cref{thm:bijectivetriplesBVmodels} concretely,
 we have to make a link between triples and Bratteli--Vershik models.
In the next section, we use the well-known fact that
 a Bratteli--Vershik model can be considered to be
 a refining sequence of tower decompositions 
 (see  \cref{prop:BV-to-KR,prop:KR-to-ordered-Bratteli-diagram}).
In this section, for concrete descriptions,
 we introduce the following notation.
Let $X = \bigcup_{1 \le i \le k(n)}\bigcup \barB(n,i)$ ($n \ge 0$)
 be a sequence of decompositions by towers $\barB(n,i)$ ($1 \le i \le k(n)$)
 with bases
 $B(n,i)$ ($1 \le i \le k(n)$) and heights $h(n,i) \ge 1$ ($1 \le i \le k(n)$)
, i.e., for every $n \ge 0$ and every $x \in X$, there exists a unique
 pair $i,j$ ($1 \le i \le k(n)$, $0 \le j < h(n,i)$)
 such that $x \in f^j(B(n,i))$.
This sequence is called a \textit{Kakutani--Rokhlin} (K--R) refinement
if the following conditions are satisfied:
\itemb
\item for every pair $m > n \ge 0$,
 each floor $f^j(B(m,i))$ ($1 \le i \le k(m), 0 \le j < h(m,i)$),
 is contained in a floor of level $n$;
\item for every pair $m > n \ge 0$, each base
 $B(m,i)$ ($1 \le i \le k(m)$) is contained in a base of level $n$; and 
\item if $\epsilon_n$ is the maximum of the diameters of the towers
 $\barB(n,i)$ ($1 \le i \le k(n)$), then $\epsilon_n \to 0$ as $n \to \infty$.
\itemn
We also assume that $k(0) = 1$ and the only tower
 $\barB(0,1)$ has base $X$ and height $h(0,1) = 1$.

\begin{nota}\label{nota:Xi-is-a-KR}
When we want to represent a K--R refinement
 $X = \bigcup_{1 \le i \le k(n)}\bigcup \barB(n,i)$ ($n \ge 0$) by $\Xi$,
 we say that 
 $\Xi : X = \bigcup_{1 \le i \le k(n)}\bigcup \barB(n,i)$ ($n \ge 0$)
 is a K--R refinement.
In this case, we denote as
\itemb
\item $B_{\Xi}(n) : = \bigcup_{1 \le i \le k(n)}B(n,i)$ for each $n \ge 0$ and
\item $B_{\Xi} := \bigcap_{n \ge 0}B_{\Xi}(n)$.
\itemn
It follows that $B_{\Xi}(n) = \bigcup_{1 \le i \le k(n)}f^{h(n,i)}(B(n,i))$
 for all $n \ge 0$.
\end{nota}
For a K--R refinement
 $\Xi : X = \bigcup_{1 \le i \le k(n)}\bigcup \barB(n,i)$ ($n \ge 0$),
 one can \textit{telescope} $\Xi$ to get only an arbitrary subsequence
 $\Xi' : X = \bigcup_{1 \le i \le k(n_l)}\bigcup \barB(n_l,i)$ ($l \ge 0$)
 with $n_0 < n_1 < n_2 < \cdots$.
Evidently, it follows that $B_{\Xi'} = B_{\Xi}$.
\begin{prop}\label{prop:K--R refinement-to-quasi-section}
Let $\Xi : X = \bigcup_{1 \le i \le k(n)}\bigcup \barB(n,i)$ ($n \ge 0$)
 be a K--R refinement of a zero-dimensional system $(X,f)$.
Then, it follows that $B_{\Xi}$ is a quasi-section.
\end{prop}
\begin{proof}
From the definition of K--R refinement, for each $n \ge 0$,
 $B_{\Xi}(n)$ is a complete section.
The conclusion follows from the definition,
 because $B_{\Xi} = \bigcap_{n \ge 0}B_{\Xi}(n)$.
\end{proof}

The following fact had been seen by Putnam in the proof of
 \cite[Lemma 3.1]{Putnam89TheCstarAlgebrasAssociatedWithMinimalHomeomorphismsOfTheCantorSet}.
\begin{thm}\label{thm:from-quasi-section-to-KR}
Let $(X,f)$ be a zero-dimensional system and $B$ be a quasi-section.
Then, there exists a K--R refinement
 $\Xi : X = \bigcup_{1 \le i \le k(n)}\bigcup \barB(n,i)$ such that
 $B = B_{\Xi}$.
\end{thm}
\begin{proof}
We take and fix a refining sequence
 $\Pcal_n$ ($n \ge 0$)
 of finite partitions of $X$ using non-empty closed and open sets;
 that is,
\itemb
\item for each $n \ge 0$, $\Pcal_n$
 is a finite set of non-empty closed and open subsets of $X$;
\item $X = \bigcup_{U \in \Pcal_n}U$ for all $n \ge 0$;
\item for all $n \ge 0$, $U \cap U' \iskuu$ for distinct $U, U' \in \Pcal_n$;
\item for each $m > n \ge 1$ and $U \in \Pcal_m$,
 there exists a $U' \in \Pcal_n$ such that
 $U \subseteq U'$; and
\item $\max \seb \diam(U) \mid U \in \Pcal_n \sen \to 0$
 as $n \to \infty$.
\itemn
We use the convention that
 $\Pcal_0 = \seb X \sen$.

We define that $k(0) = 1$ and $\barB(0,1)$ is a trivial tower
 with base floor $B(0,1) = X$ and height = 1.
Suppose that $(n-1)$th tower is built for certain $n \ge 1$ as
 $X = \bigcup_{1 \le i \le k(n-1)}\bigcup \barB(n-1,i)$.
Then, we define that
\[\Pcal'_n := \seb U \cap f^j(B(n-1,i)) \mid
 U \in \Pcal_n, 1 \le i \le k(n-1), 1 \le j < h(n-1,i) \sen.\]
We exclude the empty set from $\Pcal'$ if necessary.
We define that
 $\Ccal_n := \seb U \in \Pcal'_n \mid U \cap B \nekuu \sen$ and
 $C_n := \bigcup_{U \in \Ccal_n}U$.
Then, $C_n$ is a closed and open neighborhood of $B$.
In particular, we find that $C_n$ is a complete section.
Take an arbitrary $x \in C_n$.
Since $x \in C_n$ is recurrent with respect to $C_n$, there exists a least positive integer $h(n,x)$ such that $f^{h(n,x)}(x) \in C_n$.
We construct a finite sequence
 $s(n,x) := (U_0, U_1, U_2, \dotsc, U_{h(n,x)-1}) \in {\Pcal'_n}^{h(n,x)}$
 such that $f^j(x) \in U_j$ for all $0 \le j < h(n,x)$.
Evidently, we obtain $\sup_{x \in C_n}h(n,x) < \infty$.
Therefore, if we define that $\Scal_n := \seb s(n,x) \mid x \in C_n \sen$,
 then $\Scal_n$ is a finite set.
For each $s = (U_0,U_1,\dotsc,U_{h-1}) \in \Scal_n$,
 we define $h(s) := h$ and $B(s) := \seb x \in C_n \mid s(n,x) = s \sen$.
Evidently, for each $s \in \Scal_n$, $B(s)$ is a non-empty open set.
It is also evident that $\seb B(s) \mid s \in \Scal_n \sen$ is
 a finite partition of $C_n$.
It follows that each $B(s)$ ($s \in \Scal_n$) is also closed.
First, we show that
 the sets $f^i(B(s))$ ($s \in \Scal_n, 0 \le i < h(s))$ are mutually disjoint.
Suppose that $f^i(B(s)) \cap f^{i'}(B(s')) \nekuu$ for $(s,i) \ne (s',i')$ 
 with $s, s' \in \Scal_n$, with $1 \le i < h(s)$, and with $1 \le i' < h(s')$.
Then, take and fix an $x \in f^i(B(s)) \cap f^{i'}(B(s'))$.
Without loss of generality, we can assume that $i < i'$.
Let $y = f^{-i}(x) \in C_n$ and $z = f^{-i'}(x) \in C_n$.
Then, we get $f^{i'-i}(z) = y$.
Thus, we obtain $f^{i'-i}(B(s')) \cap C_n \nekuu$.
Because $i'-i < h(s')$, we obtain a contradiction.
Therefore, we get mutually disjoint towers $\barB(s)$ ($s \in \Scal_n$).
We have to show that $X = \bigcup_{s \in \Scal_n}\bigcup \barB(s)$.
Take and fix an $x \in X$ arbitrarily.
Then, it follows that there exists the least
 integer $i \ge 0$
 such that $f^{-i}(x) \in C_n$.
Suppose that $i = 0$.
Then, we get $x \in C_n$.
Because $\seb B(s) \mid s \in \Scal_n \sen$ covers
 $C_n$, we get $x \in \bigcup_{s \in \Scal_n}\bigcup \barB(s)$, as desired.
Thus, we assume that $i > 0$.
Let $y = f^{-i}(x)$.
Then, because $h(n,y)$ is the least positive integer
 such that $f^{h(n,y)}(y) \in C_n$,
 we get $i < h(n,y)$.
Thus, we obtain $x \in \bigcup \barB(s(n,y))$, as desired.
From the construction, every floor of $\bigcup_{s \in \Scal_n}\bigcup \barB(s)$
 is contained in a floor of level $n-1$.
Additionally, every base floor of level $n$ is contained 
 in a base floor of level $n-1$.
The maximum diameters of the towers
 $\barB(s)$ ($s \in \Scal_n$)
 tend to zero as $n \to \infty$ owing to the last condition of $\Pcal_n$.
Thus, we get a decomposition by towers
 $X = \bigcup_{s \in \Scal_n}\bigcup \barB(s)$.
We rewrite this as
 $X = \bigcup_{1 \le i \le k(n)}\bigcup \barB(n,i)$, where $k(n) = \abs{\Scal_n}$.
Through this induction, we obtain a K--R refinement
 $\Xi : X = \bigcup_{1 \le i \le k(n)}\bigcup \barB(n,i) (n \ge 0)$.
Finally, from the construction, we obtain $B = B_{\Xi}$.
\end{proof}

\begin{rem}\label{rem:telescoping-KR-refinement}
Let $(X,f,B)$ be a triple of a zero-dimensional system 
 $(X,f)$ and a quasi-section $B \subseteq X$.
Then, it follows that there exists a unique equivalence class of
 K--R refinements that are generated by telescopings.
In other words,
 if $\Xi_1 : X = \bigcup_{1 \le i \le k_1(n)}\overline{B_1}(n,i)$
 and $\Xi_2 : X = \bigcup_{1 \le i \le k_2(n)}\overline{B_2}(n,i)$ are
 K--R refinements such that $B = B_{\Xi_1} = B_{\Xi_2}$,
 then there exists a K--R refinement
 $\Xi : X = \bigcup_{1 \le i \le k(n)}\bigcup \barB(n,i)$ (with $B = B_{\Xi}$)
 that has two telescopings $\Xi'_1$ and $\Xi'_2$ with
 $\Xi'_1$ being a telescoping of $\Xi_1$
 and $\Xi'_2$ being a telescoping of $\Xi_2$.
Thus, for a fixed triple $(X,f,B)$, all K--R refinements $\Xi$ with
 $B = B_{\Xi}$ are equivalent to each other.
Here, we do not introduce the equivalence relation of K--R refinements
 on different triples.
\end{rem}

\section{Bratteli--Vershik models and main basic results}\label{sec:main-basic-results}
In this section, we define general Bratteli--Vershik models and
 make a link with K--R refinements.
We describe \cref{prop:BV-to-KR,prop:KR-to-ordered-Bratteli-diagram}
 to check that these are equivalent notions.
Therefore, owing to \cref{thm:from-quasi-section-to-KR},
 we shall see that we have already
 done a concrete description of \cref{thm:bijectivetriplesBVmodels}.
\subsection{Bratteli--Vershik models}\label{subsec:Bratteli--Vershik-models}
A \textit{Bratteli diagram} is an infinite directed graph $(V,E)$,
 where $V$ is the vertex set and $E$ is the edge set.
The vertex set $V$ is decomposed into non-empty finite sets
 $V = V_0 \cup V_1 \cup V_2 \cup \dotsb$,
 where $V_0 = \seb v_0 \sen$ is a single point.
The edge set $E = E_1 \cup E_2 \cup \dotsb$ is also decomposed into non-empty
 finite sets.
Each $E_n$ is a set of edges from $V_{n-1}$ to $V_n$ for each $n > 0$.
Therefore, there exist two maps $r,s : E \to V$ such that $s : E_n \to V_{n-1}$
 and $r:E_n \to V_n$ for all $n \bpi$,
i.e., the \textit{source map} and the \textit{range map}, respectively.
Moreover, $s^{-1}(v) \nekuu$ for all $v \in V$ and
$r^{-1}(v) \nekuu$ for all $v \in V \setminus V_0$.
We say that $u \in V_{n-1}$ is connected to $v \in V_{n}$ if there
 exists an edge $e \in E_n$ such that $s(e) = u$ and $r(e) = v$.

We consider finite or infinite path spaces for a Bratteli diagram.
For each $0 \le n < m$, a sequence of edges
 $p = (e_{n+1},e_{n+2},\dotsc,e_m) \in \prod_{n < i \le m}E_i$ 
 with $r(e_i) = s(e_{i+1})$ for all $n < i < m$ is called a \textit{path}.
A path $p = (e_{n+1},e_{n+2},\dotsc,e_m)$ extends from one vertex $v \in V_n$
 to another vertex $v' \in V_m$ if $v = s(e_{n+1})$ and $v' = r(e_m)$.
For each $n < m$, we define
 $E_{n,m} := \seb p \in \prod_{n < i \le m}E_i \mid p \text{ is a path.} \sen.$
For $p = (e_{n+1},e_{n+2},\dotsc,e_m) \in E_{n,m}$,
 the source map $s : E_{n,m} \to V_n$ and the range map $r : E_{n,m} \to V_m$
 are defined by $s(p) = s(e_{n+1})$ and $r(p) = r(e_{m})$, respectively.
\begin{defn}\label{defn:simple}
A Bratteli diagram is \textit{simple} if for every
 $n \ge 0$ there exists an $m > n$ such that
 all vertices in $V_n$ are linked with all vertices in $V_m$ by
 paths in $E_{n,m}$.
\end{defn}
For each $n \ge 0$, an infinite path $p = (e_{n+1},e_{n+2},\dotsc)$ 
 is also defined.
For each $n \ge 0$, $E_{n,\infty}$ denotes
 the set of all infinite paths from $V_n$.
For $p = (e_{n+1},e_{n+2},\dotsc) \in E_{n,\infty}$,
 the source map $s : E_{n,\infty} \to V_n$ is defined as $s(p) = s(e_{n+1})$.
For a finite or infinite path $p = (e_{n+1},e_{n+2},\dotsc)$,
 $p(i) := e_i$ for each $n < i$ if $e_i$ is defined.
In particular, we have defined the set
 $E_{0,\infty}$.
We consider $E_{0,\infty}$ with the product topology.
Under this topology, it is a compact zero-dimensional space.

Let $(V,E)$ be a Bratteli diagram.
We say that $(V,E,\ge)$ is an \textit{ordered Bratteli diagram} if
 a partial order $\ge$ is defined on $E$ such that 
 $e, e' \in E$ are comparable if and only if $r(e) = r(e')$.
 Thus, we have a linear order on each set $r^{-1}(v)$ for each $v \in \Vp$.
The edges $r^{-1}(v)$ are numbered from $1$ to $\abs{r^{-1}(v)}$, and
 the maximal (resp. minimal) edge is denoted
 by $e(v,\max)$ (resp. $e(v,\min)$).
Let $E_{\max}$ and $E_{\min}$ denote the sets of maximal and minimal
 edges, respectively.
Thus, we obtained a partial order on path spaces by the 
 lexicographic order.
Let $(V,E,\ge)$ be an ordered Bratteli diagram.
For each $0 < n < m$ and $v \in V_m$,
the set $\seb p \in E_{n,m} \mid r(p) = v \sen$ 
 is linearly ordered by the lexicographic order, 
 i.e., for $p \ne q \in E_{n,m}$ with $r(p) = r(q)$,
 $p < q$ if and only if
 $p(k) < q(k)$ with the maximal $k \in [n+1,m]$ such that $p(k) \ne q(k)$.
Let $n_0 = 0 < n_1 < n_2 < \cdots$ be an increasing sequence of integers.
Then, by replacing $E_{n_k,n_{k+1}}$ instead of the subgraph from $V_{n_k}$
 to $V_{n_{k+1}}$ for all $k \ge 0$,
 we obtain a new ordered Bratteli diagram $(V',E',\ge)$
 called \textit{telescoping}, where $V' = \bigcup_{k \ge 0}V_{n_k}$
 and $E' = \bigcup_{k \ge 1}E_{n_{k-1},n_{k}}$ (see \cite[Definition 3.2]{GPS_95TopologicalOrbitEquivAndCstarCrossedProduct}).
For each $n \bni$, suppose that
 $p, p' \in E_{n,\infty}$ are distinct cofinal paths,
 i.e., there exists a $k > n$
 such that $p(k) \ne p'(k)$, and for all $l > k$, $p(l) = p'(l)$.
We define the lexicographic order $p < p'$ if and only if $p(k) < p'(k)$.
In particular, we have defined the lexicographic order on $E_{0,\infty}$.
This is a partial order, and $p,q \in E_{0,\infty}$ is comparable if 
 and only if $p$ and $q$ are cofinal.
Let $(V,E,\ge)$ be an ordered Bratteli diagram.
We define
\begin{fleqn}[20mm]
\begin{align*}
\ & E_{0,\infty,\min} := \seb p \in E_{0,\infty} \mid 
 p(k) \in E_{\min} \myforall k \sen, \myand\\
\ & E_{0,\infty,\max} := \seb p \in E_{0,\infty} \mid 
 p(k) \in E_{\max} \myforall k \sen.
\end{align*} 
\end{fleqn}
\begin{defn}\label{defn:oBd-properly-ordered}
An ordered Bratteli diagram $(V,E,\ge)$ is \textit{properly ordered}
 if $(V,E)$ is simple and
 $\abs{E_{0,\infty,\min}} = \abs{E_{0,\infty,\min}} = 1$.
\end{defn}
For each $p \in E_{0,\infty} \setminus E_{0,\infty,\max}$,
 there exists the least $p' > p$ with respect to the lexicographic order.
Thus, we can consider the Vershik map
\[\psi : E_{0,\infty} \setminus E_{0,\infty,\max} \to E_{0,\infty}\]
such that $\psi(p)$ is the least element with $p < \psi(p)$.
In certain classes of ordered Bratteli diagrams, the Vershik map
 can be extended to $\psi : E_{0,\infty} \to E_{0,\infty}$ continuously.
We note that this extension may not be unique.

We say that $(V,E,\ge)$ admits a \textit{continuous Vershik map}
 if the Vershik map $\psi$ can be extended to 
 $\psi : E_{0,\infty} \to E_{0,\infty}$
 with $\psi(E_{0,\infty,\max}) \subseteq E_{0,\infty,\min}$ continuously.
Specifically, we note that
\itemb
\item $\psi$ is surjective
 if and only if $\psi(E_{0,\infty,\max}) = E_{0,\infty,\min}$,
\item $\psi$ is injective if and only if
 $\psi$ is injective on $E_{0,\infty,\max}$, and
\item if $\abs{\psi^{-1}(x)} \ne 1$, then $x \in E_{0,\infty,\min}$.
\itemn
\begin{defn}\label{defn:BV-model}
If $(V,E,\ge)$ admits a Vershik map $\psi$ that is a homeomorphism,
 then we say that $(V,E,\ge,\psi)$ is a \textit{Bratteli--Vershik model}.
We say that the Bratteli--Vershik model is \textit{properly ordered}
 if $(V,E,\ge)$ is properly ordered.
\end{defn}
Let $(V,E,\ge,\psi)$ be a properly ordered Bratteli--Vershik model.
Then, it follows that $\abs{E_{0,\infty,\min}} = \abs{E_{0,\infty,\max}} = 1$
 by the definition.

\begin{defn}
We say that two
 Bratteli--Vershik models $(V,E,\ge,\psi)$ and $(V',E',\ge',\psi')$
 (with the path spaces $E_{0,\infty}$ and $E'_{0,\infty}$ respectively)
are \textit{topologically conjugate} if there exists a homeomorphism
 $\fai : E_{0,\infty} \to E'_{0,\infty}$ that satisfies
 $\fai \circ \psi = \psi' \circ \fai$ and
 $\fai(E_{0,\infty,\min}) = E'_{0,\infty,\min}$.
\end{defn}

If $(E',V',\ge)$ is a telescoping, it is natural to obtain an isomorphism
 $E_{0,\infty} = E'_{0,\infty}$ with $E_{0,\infty,\min} = E'_{0,\infty,\min}$
 and with $E_{0,\infty,\max} = E'_{0,\infty,\max}$.
Thus, we get an isomorphic zero-dimensional system $(E'_{0,\infty},\psi)$.
Two ordered Bratteli diagrams $(V^1,E^1,\ge)$ and $(V^2,E^2,\ge)$ are 
 \textit{equivalent} if there exists an ordered Bratteli diagram 
 $(V,E,\ge)$ such that
\itemb
\item $(V^1,E^1,\ge)$ and $(V,E,\ge)$ have a common telescoping and
\item $(V^2,E^2,\ge)$ and $(V,E,\ge)$ have a common telescoping.
\itemn
We note that, even if the two ordered Bratteli diagrams $(V^1,E^1,\ge)$ and $(V^2,E^2,\ge)$ are 
 \textit{equivalent},
 we cannot yet discuss the isomorphism of the Vershik maps,
 because we need to consider not only the existence but also the uniqueness of
 the homeomorphic extension of the Vershik maps.
In general, a zero-dimensional system $(X,f)$ is said to have
 a Bratteli--Vershik model if $(X,f)$ is topologically conjugate to
 $(E_{0,\infty},\psi)$ for a certain Bratteli--Vershik model $(V,E,\ge,\psi)$.

\subsection{A link with K--R refinements}\label{subsec:alinkKR}

In this subsection, we make a link with K--R refinements.
The next notation combines
 the ordered Bratteli diagrams with the notion of K--R refinements.

\begin{nota}\label{nota:paths-from-vertices}
Let $(V,E,\ge,\psi)$ be a Bratteli--Vershik model, $n > 0$, and $v \in V_n$.
We abbreviate $P(v) := \seb p \in E_{0,n} \mid r(p) = v \sen$.
We define $h(v) := \abs{P(v)}$ and
 write $P(v) = \seb p(v,0) < p(v,1) < \dotsb < p(v,h(v)-1) \sen$.
Let $U(v,j) := \seb x = (e_{x,1}, e_{x,2},\dotsc) \in E_{0,\infty} \mid 
 (e_{x,1}, e_{x,2}, \dotsc, e_{x,n}) = p(v,j) \sen$
  for all $0 \le j < h(v)$.
We denote $B(v) := U(v,0)$ and $\barB(v) := \seb U(v,j) \mid 0 \le j < h(v) \sen$.
Then, for any Vershik map $\psi$, $\barB(v)$
 is a tower, i.e., $\psi^j(B(v)) = U(v,j)$ for all $0 \le j < h(v)$.
Clearly, we get a decomposition by towers
 $X = \bigcup_{v \in V_n}\bigcup \barB(v)$ ($n \ge 0$).
\end{nota}

\begin{prop}[From Bratteli--Vershik model to K--R refinement]\label{prop:BV-to-KR}
Let $(V,E,\ge,\psi)$ be a Bratteli--Vershik model.
Then, the decomposition by towers
 $E_{0,\infty} = \bigcup_{v \in V_n}\bigcup \barB(v)$ ($n \ge 0$)
 is a K--R refinement.
\end{prop}
\begin{proof}
First, let $m > n \ge 0$, $v \in V_m$, and $0 \le j < h(v)$.
Then, $\psi^j(B(v)) = U(v,j)$ in the way of 
 \cref{nota:paths-from-vertices}.
Therefore, we get
 $U(v,j) = \seb x = (e_{x,1}, e_{x,2},\dotsc) \in E_{0,\infty} \mid 
 (e_{x,1}, e_{x,2}, \dotsc, e_{x,m}) = p(v,j) \sen$.
If $p(v,j) = (e_1, e_2, \dotsc, e_m)$, then
 there exists $v' \in V_n$ and $0 \le j' < h(v')$ such that
 $(e_1,e_2,\dotsc, e_n) = p(v',j')$.
Thus, we get
 $U(v,j) \subseteq U(v',j')$, as desired.
Second, let $m > n \ge 0$ and $v \in V_n$.
Then, $B(v) = \seb x = (e_{x,1}, e_{x,2},\dotsc) \in E_{0,\infty} \mid 
 (e_{x,1}, e_{x,2}, \dotsc, e_{x,m}) = p(v,0) \sen$.
If $p(v,0) = (e_1,e_2,\dotsc, e_m)$, then
 we get $e_i \in E_{\min}$ ($1 \le i \le m$).
Thus, if $v' \in r(e_n)$, then we obtain $B(v) \subseteq B(v')$,
 as desired.
Evidently, because the topology in $E_{0,\infty}$ is induced by the product topology,
 the third condition follows.
\end{proof}

The converse is stated in the following.
\begin{prop}[From K--R refinement to Bratteli--Vershik model]\label{prop:KR-to-ordered-Bratteli-diagram}
Let $(X,f)$ be a zero-dimensional system and
 $\Xi : X = \bigcup_{1 \le i \le k(n)}\bigcup \barB(n,i)$ be a K--R refinement.
Then, we can obtain a Bratteli--Vershik model $(V,E,\ge,\psi)$
 such that
\itemb
\item there exists a topological conjugacy
 $\fai : (X,f) \to (E_{0,\infty},\psi)$,
\item $\fai(B_{\Xi}) = E_{0,\infty,\min}$,
\item $\fai(f^{-1}(B_{\Xi})) = E_{0,\infty,\max}$, and
\item the decomposition by towers $E_{0,\infty} = \bigcup_{v \in V_n}\bigcup \barB(v)$ ($n \ge 0$) matches $\Xi$ by $\fai$.
\itemn
\end{prop}
\begin{proof}
First, we construct a Bratteli--Vershik model $(V,E,\ge, \psi)$
 with which the $\psi : E_{0,\infty} \to E_{0,\infty}$
 is topologically conjugate to $f: X \to X$.
Let $V_0$ be a one-point set that consists of the unique tower $\barB(0,1)$
 with height $1$ and base $B(0,1) = X$.
Suppose that we have constructed $V_{n-1}$ ($n \ge 1$) to be the set of towers
 of level $n-1$.
Then, $V_n$ is the set of towers of level $n$, i.e.,
 $V_n = \seb \barB(n,i) \mid 1 \le i \le k(n) \sen$.
Next, we have to construct the edge set $E_n$.
To do so, fix $v = \barB(n,i) \in V_n$ ($1 \le i \le k(n)$) arbitrarily.
Then, for an arbitrary point $x \in B(n,i)$,
 the finite orbit $x, f(x), f^2(x), \dotsc, f^{h(n,i)-1}(x)$
 passes through the same sequence 
 $\barB(n-1,i_1), \barB(n-1,i_2), \dotsc, \barB(n-1,i_{a(v)})$, 
 of towers of level $n-1$ successively. 
We make $a(v)$ edges $e(v,j)$ ($1 \le j \le a(v)$)
 from $\barB(n-1,i_j)$ to $v = \barB(n,i)$ for each $1 \le j \le a(v)$
 with the linear order such that $e(v,j) \le e(v,j')$ if and only if
 $j \le j'$.
Thus, $s(e(v,j)) = \barB(n-1,i_j)$ and $r(e(v,j)) = v$ for each $v \in V_n$ 
 and each $1 \le j \le a(v)$.
We define
 $E_n := \seb e(v,j) \mid v \in V_n, 1 \le j \le a(v) \sen$.
Thus, we have constructed an ordered Bratteli diagram $(V,E,\ge)$.
Next, we make a map $\fai : X \to E_{0,\infty}$ as follows.
Take an $x \in X$ and $n \ge 1$ arbitrarily.
Then, there exists a unique floor $f^{j(x)}(B(n,i(x))) \ni x$
 with $1 \le i(x) \le k(n)$ and $0 \le j(x) < h(n,i(x))$.
Let $v(x) := \barB(n,i(x)) \in V_n$.
Take the unique $y \in B(n,i(x))$ such that $f^{j(x)}(y) = x$.
Then, the finite orbit $y, f(y), f^2(y), \dots, f^{j(x)}(y)$ passes
 the sequence $\barB(n-1,i_1), \barB(n-1,i_2), \dotsc, \barB(n-1,i_{b(x)})$
 of towers successively and we get $x \in \barB(n-1,i_{b(x)})$.
We define $e_{x,n} := e(v(x),b(x)) \in E_n$.
Because $n \ge 1$ is arbitrary,
 we define $\fai(x) := (e_{x,1},e_{x,2},\dotsc) \in E_{0,\infty}$.
Thus, we defined a map $\fai : X \to E_{0,\infty}$.
The continuity of $\fai$ is evident from the construction.
The injectivity of $\fai$ follows from the third condition of K--R refinement.
To show the surjectivity,
 take an arbitrary $p = (e_1,e_2,\dotsc) \in E_{0,\infty}$, and
let $n \ge 1$.
Then, the sequence $(e_1,e_2,\dotsc,e_n)$ indicates a unique floor $U(p,n)$ in 
 the tower $r(e_n)$.
Because $U(p,1) \supseteq U(p,2) \supseteq \cdots$,
 the third condition of K--R refinement
 implies that there exists a unique $x \in X$
 such that $\seb x \sen = \bigcap_{n \ge 1}U(p,n)$.
It is now evident that $\fai(x) = p$.
Thus, we find that $\fai$ is a homeomorphism.
We define 
 $\psi := \fai \circ f \circ \fai^{-1} : E_{0,\infty} \to E_{0,\infty}$.
From the construction, it follows that $\fai(B_\Xi) = E_{0,\infty,\min}$.
It is clear that if $x \notin f^{-1}(B_\Xi)$,
 then $x$ is not on the top floor
 at a certain level $n$, i.e., $\fai(x) \notin E_{0,\infty,\max}$.
Thus, evidently, we get
 $\fai(X \setminus f^{-1}(B_\Xi))
 \subseteq E_{0,\infty} \setminus E_{0,\infty,\max}$, i.e., we get
 $\fai(f^{-1}(B_\Xi)) \supseteq E_{0,\infty,\max}$.
Conversely, let $x \in f^{-1}(B_\Xi)$.
Then, $x$ is on the top floor of a certain tower at each level $n$.
This implies that $e_{x,n} \in E_{\max}$ for all $n \ge 1$.
Therefore, we get $\fai(x) \in E_{0,\infty,\max}$.
Thus, we obtain $\fai(f^{-1}(B_\Xi)) = E_{0,\infty,\max}$.
It follows that $\psi(E_{0,\infty,\max}) = E_{0,\infty,\min}$.
From the construction, the third condition of this proposition is satisfied.
It is now evident that $\psi$ is determined by the lexicographic order
 on the set $E_{0,\infty} \setminus E_{0,\infty,\max}$.
\end{proof}

Let $(V,E,\ge,\psi)$ be a properly ordered Bratteli--Vershik model.
Then, it follows that $(E_{0,\infty},\psi)$ is minimal.
This is obtained by simplicity:
 for each $n \ge 0$ there exists an $m > n$
 such that
 all the towers of level $m$ wind all towers of level $n$.
Conversely, a minimal system has a properly ordered Bratteli--Vershik model
 (see \cite{HERMAN_1992OrdBratteliDiagDimGroupTopDyn}).

\subsection{Subsystems and Bratteli--Vershik models}
Let $(X,f)=(E_{0,\infty},\psi)$ with some 
 Bratteli--Vershik model $(V,E,\ge,\psi)$, and
 $\Xi : X = \bigcup_{1 \le i \le k(n)}\bigcup \barB(n,i)$ be the
 corresponding K--R refinement.
Let $(Y,f|_Y)$ be a subsystem of $(X,f)$.
We shall construct the corresponding Bratteli--Vershik model with
 a thorough check.
If we write as $B = E_{0,\infty,\min}$, then 
 we get an isomorphism $(X,f,B) = (E_{0,\infty},\psi,E_{0,\infty,\min})$.
By the correspondence
 with respect to \cref{prop:BV-to-KR,prop:KR-to-ordered-Bratteli-diagram},
We consider
 $V_n = \seb \barB(n,i) \mid 1 \le i \le k(n) \sen$ for each $n \ge 0$,
 and $V = \bigcup_{n \ge 0}V_n$.
The edge set $E$ is also constructed
 according to \cref{prop:KR-to-ordered-Bratteli-diagram}.
Suppose that $B(n,i)$ is a base floor of level $n$ with $B(n,i) \cap Y \nekuu$.
Then, it is evident that $B' := B(n,i) \cap Y$ is a base floor
 of the tower $\seb B', f(B'), \dotsc, f^{h(n,i)-1}(B')\sen$
 with respect to $(Y,f|_Y)$.
We renumber $1 \le i \le k(n)$ if necessary
 to get some $k(Y,n) \le k(n)$ with
 $B(n,i) \cap Y \nekuu \Leftrightarrow i \le k(Y,n)$.
We define $B(Y,n,i) := B(n,i) \cap Y$ for $1 \le i \le k(Y,n)$.
Thus, we get a K--R refinement
 $\Xi(Y) : \bigcup_{1 \le i \le k(Y,n)}\bigcup \barB(Y,n,i)$
 of $(Y, f|_Y)$.
Note that for $m > n$ and $1 \le j \le k(Y,n)$, a tower $\barB(m,j)$ 
 passes through $\barB(n,i)$ if and only if 
 the tower $\barB(Y,m,j)$ passes through a tower $\barB(Y,n,i)$.
Precisely, for $m > n$, $f^h(B(m,j))\cap B(n,i) \nekuu$
 $(1 \le j \le k(Y,m), 0 \le h < h(m,j))$
 if and only if $1 \le i \le k(Y,n)$ and 
 $f^h(B(Y,m,j)) \cap B(Y,n,i) \nekuu$.
Let $V(Y)_n := \seb \barB(n,i) \mid 1 \le i \le k(Y,n) \sen$
 for all $n \ge 0$ and $V(Y) := \bigcup_{n \ge 0}V(Y)_n$.
Thus, we get that $V(Y) \subseteq V$.
For $n \ge 1$,
 we define $E(Y)_n := \bigcup_{v \in V(Y)_n}r^{-1}(v) \subseteq E_n$.
 and $E(Y) := \bigcup_{n \ge 1}E(Y)_n$.
Thus, we get that $E(Y) \subseteq E$.
Because $r^{-1}(v) \subseteq E$ for all $V(Y)\setminus V(Y)_0$,
 we can use the same order on $E(Y)$ as $E$.
We define as $\psi(Y) := \psi|_{E(Y)_{0,\infty}}$.
Thus, we get a Bratteli--Vershik model $(V(Y),E(Y),\ge,\psi(Y))$.
It is evident that $(Y,f|_Y) = (E(Y)_{0,\infty},\psi(Y))$.
It is also evident that $(Y,f|_Y, B \cap Y)$
 is isomorphic to $(E(Y)_{0,\infty},\psi(Y),E(Y)_{0,\infty,\min})$.
\begin{defn}\label{defn:BV-quasi-simple}
We say that
 a Bratteli--Vershik model $(V,E,\ge, \psi)$ is \textit{quasi-simple}
 if for each $M \in \Mcal_{\psi}$, $(V(M), E(M),\ge, \psi(M))$
 is a properly ordered Bratteli--Vershik model.
\end{defn}

\begin{thm}\label{thm:quasi-simple}
A Bratteli--Vershik model $(V,E,\ge,\psi)$ is quasi-simple
 if and only if the set $E_{0,\infty,\min}$ is a quasi-simple
 quasi-section.
\end{thm}
\begin{proof}
Let $(V,E,\ge,\psi)$ be a Bratteli--Vershik model.
We write as $(X,f) = (E_{0,\infty},\psi)$ and $B = E_{0,\infty,\min}$.
Suppose that $(V,E,\ge,\psi)$ is quasi-simple.
Then, we have to show that $B$ is a quasi-simple quasi-section.
To see this, let $M \in \Mcal_f$.
Then, $(V(M),E(M),\ge,\psi(M))$ is a properly ordered Bratteli--Vershik
 model.
It follows that $(V(M),E(M))$ is simple and
 $\abs{E(M)_{0,\infty,\min}} = \abs{E(M)_{0,\infty,\max}} = 1$.
Because $(M,f|_M, B \cap M)$
 is isomorphic to $(E(M)_{0,\infty},\psi(M),E(M)_{0,\infty,\min})$,
 $B \cap M$ is a single point, as desired.
Conversely, suppose that $B$ is a quasi-simple quasi-section of 
 $(X,f,B)$.
Then, it follows that $\abs{B \cap M} = 1$ for each $M \in \Mcal_f$.
Let $M \in \Mcal_f$.
We have to show that $(V(M),E(M),\ge,\psi(M))$ is properly ordered.
The simplicity of $(V(M),E(M))$ follows from the fact that $M$ is a
 minimal set.
Because $(M,f|_M, B \cap M)$
 is isomorphic to $(E(M)_{0,\infty},\psi(M),E(M)_{0,\infty,\min})$,
 it follows that $\abs{E(M)_{0,\infty,\min}} = 1$.
Finally, $\abs{E(M)_{0,\infty,\max}} = 1$ follows from the fact 
 that $\psi$ is a homeomorphism.
\end{proof}

The quasi-simpleness condition does not imply
 that $E_{0,\infty,\min}$ is a basic set
 (see \cref{rem:quasi-simple-may-not-closing} below).

\subsection{The proofs}\label{subsec:proofofbasicresults}
In this subsection, we finish our complete proofs of \cref{thm:bijectivetriplesBVmodels,thm:bijectivebasicsetsBVmodelswithclosingproperty,thm:bijectiveqsbasicsetsBVmodelswithssclosingproperty}.

\vspace{3mm}

\noindent \textit{Proof of \cref{thm:bijectivetriplesBVmodels}}.

Let $(V,E,\ge,\psi)$ be a Bratteli--Vershik model.
Then, it is evident that $E_{0,\infty,\min}$ is a quasi-section.
To show the converse,
 let $(X,f,B)$ be a triple of a zero-dimensional system and a quasi-section.
Combining
 \cref{thm:from-quasi-section-to-KR,prop:KR-to-ordered-Bratteli-diagram},
 we obtain a Bratteli--Vershik model
 $(V,E,\ge,\psi)$
 such that $B = E_{0,\infty,\min}$ with respect to a topological conjugacy.
\qed

\vspace{5mm}

To see the correspondence between Bratteli--Vershik models and basic sets,
 we list a definition and a result from our previous paper
\cite{Shimomura_2020AiMBratteliVershikModelsAndGraphCoveringModels}
 as follows:
\begin{defn}\label{defn:keeping-constant}
Let $(V,E)$ be a Bratteli diagram and $n \bni$.
We say that an infinite path $(e_{n+1},e_{n+2},\dotsc) \in E_{n,\infty}$
 is \textit{constant}
 if $\abs{r^{-1}(r(e_i))} = 1$ for all $i > n$.
A Bratteli--Vershik model $(V,E,\ge,\psi)$
 has the \textit{closing property} 
 if, for every constant path $(e_{n+1},e_{n+2}, \dotsc) \in E_{n,\infty}$
 with $n \bni$,
 the set $\bigcap_{m > n}\bigcup \barU(s(e_m))$ is a periodic orbit
 (of least period $h(s(e_{n+1}))$).
\end{defn}

\begin{rem}\label{rem:quasi-simple-may-not-closing}
Even if $(V,E,\ge, \psi)$ is quasi-simple,
 it may not have the closing property,
 because a constant path $(e_{n+1},e_{n+2},\dotsc) \in E_{n,\infty}$
 may correspond to a finite segment of an aperiodic orbit that is not 
 contained in any minimal set.
Therefore, this orbit enters $E_{0,\infty,\min}$ multiple times.
\end{rem}

We list a theorem from \cite{Shimomura_2020AiMBratteliVershikModelsAndGraphCoveringModels} without a proof.
\begin{thm}[Theorem 4.19 in \cite{Shimomura_2020AiMBratteliVershikModelsAndGraphCoveringModels}]\label{thm:closing-basic}
A Bratteli--Vershik model $(V,E,\ge,\psi)$ has the closing property
 if and only if the set $E_{0,\infty,\min}$ is a basic set.
\end{thm}

\vspace{3mm}

\noindent \textit{Proof of \cref{thm:bijectivebasicsetsBVmodelswithclosingproperty}}

Let $(V,E,\ge,\psi)$ be a Bratteli--Vershik model with closing property.
Then, by \cref{thm:closing-basic}, it follows that
 $E_{0,\infty,\min}$ is a basic set.
To show the converse,
 let $(X,f,B)$ be a triple of a zero-dimensional system and a basic set.
Combining
 \cref{thm:from-quasi-section-to-KR,prop:KR-to-ordered-Bratteli-diagram},
 we obtain a Bratteli--Vershik model
 $(V,E,\ge,\psi)$
 such that $B = E_{0,\infty,\min}$ with respect to a topological conjugacy.
Because $E_{0,\infty,\min}$ is a basic set, it follows that
 $(V,E,\ge,\psi)$ has closing property by \cref{thm:closing-basic}.
\qed

\vspace{3mm}

\noindent \textit{Proof of \cref{thm:bijectiveqsbasicsetsBVmodelswithssclosingproperty}}

Let $(V,E,\ge,\psi)$ be a quasi-simple
 Bratteli--Vershik model.
Then, in the zero-dimensional system $(E_{0,\infty},\psi)$,
 it follows that $E_{0,\infty,\min}$ is a quasi-simple quasi-section
 by \cref{thm:quasi-simple}.
To show the converse,
 let $(X,f,B)$ be a triple of a zero-dimensional system and a quasi-simple
 quasi-section.
Combining
 \cref{thm:from-quasi-section-to-KR,prop:KR-to-ordered-Bratteli-diagram},
 we obtain a Bratteli--Vershik model
 $(V,E,\ge,\psi)$
 such that $(E_{0,\infty}, \psi, E_{0,\infty,\min})$ is topologically 
 conjugate to $(X,f,B)$.
Now, the conclusion follows from \cref{thm:quasi-simple}.
\qed

\vspace{3mm}

Evidently, the same argument as in \cite[Theorem 4.4]{HERMAN_1992OrdBratteliDiagDimGroupTopDyn} can be applied in the case of K--R refinements of 
 a triple $(X,f,B)$ of a zero-dimensional system
 $(X,f)$ and a quasi-section $B \subseteq X$ (cf. \cref{rem:telescoping-KR-refinement}).
Therefore, the correspondence in the previous theorem can be refined 
 to the level of the ordered Bratteli diagrams partially as the following
 (see also \cite[Theorem 6.6]{BezuglyiNiuSun_2021CstarAlgOfaCantorSysWithFinitelyManyMinStructKTtheAndIndexMap}).
\begin{thm}\label{thm:triple-implies-equivalent-ordered-Bratteli-diagrams}
Let $(X,f)$ be a zero-dimensional system and $B$ be a quasi-section.
The equivalence class of the ordered Bratteli diagram constructed 
 from the triple $(X,f,B)$ does not depend on the choice of K--R refinements.
\end{thm}
\begin{proof}
The proof is evident from \cref{rem:telescoping-KR-refinement}.
\end{proof}

An ordered Bratteli diagram does not have 
 the continuous Vershik map in general
 (cf., for example, \cite{BezuglyiYassawi2017OrdersThatYieldHomeoOnBratteliDiagrams,BezuglyiKwiatkowskiYassawi_2014PerfectOrderingsOnFiniteRankBratteliDiagrams}).
Therefore, the converse statement to \cref{thm:triple-implies-equivalent-ordered-Bratteli-diagrams} must be partial.
To get some bijective correspondences, we have to clarify the subclass
 of the triples and also that of the ordered Bratteli diagrams.
To do this, we need the work by Downarowicz and Karpel
 \cite{DownarowiczKarpel_2019DecisiveBratteliVershikmodels}.
We present this in \cref{sec:decisiveness}.

\section{Decisiveness, quasi-sections, and basic sets}\label{sec:decisiveness}

In \cite{DownarowiczKarpel_2019DecisiveBratteliVershikmodels},
 Downarowicz and Karpel provided the following definition:
\begin{defn}\label{defn:decisive}
An ordered Bratteli diagram $(V,E,\ge)$ is \textit{decisive}
 if the Vershik map
 $\psi: E_{0,\infty} \setminus E_{0,\infty,\max} \to E_{0,\infty}$
 prolongs in a unique manner to a homeomorphism
 $\psi:E_{0,\infty} \to E_{0,\infty}$.
\end{defn}
In this study, for a decisive ordered Bratteli diagram $(V,E,\ge)$,
 the unique Bratteli--Vershik model
 $(V,E,\ge,\psi)$ is also considered to be \textit{decisive}.
They also provided the following definition:
\begin{defn}\label{defn:Bratteli--Vershikizable}
A zero-dimensional system $(X, f)$
 is called \textit{Bratteli--Vershikizable}
 if it is conjugate to $(E_{0,\infty},\psi)$
 for a decisive ordered Bratteli diagram $(V,E,\ge)$.
\end{defn}
Moreover, they have shown that a zero-dimensional system $(X,f)$
 is Bratteli--Vershikizable if and only if
 either the set of aperiodic points is dense
 or its closure misses one periodic orbit
 (see \cite[Theorem 3.1]{DownarowiczKarpel_2019DecisiveBratteliVershikmodels}).
They have given a fundamental observation of decisiveness in
 \cite[Proposition 1.2]{DownarowiczKarpel_2019DecisiveBratteliVershikmodels}.
We present this as follows:
\begin{prop}[Proposition 1.2 by Downarowicz and Karpel \cite{DownarowiczKarpel_2019DecisiveBratteliVershikmodels}]\label{prop:basic-concept-from-DK}
An ordered Bratteli diagram $(V,E,\ge)$ is decisive if and only if
 the following two conditions hold:

\noindent \textup{(A)} A homeomorphic Vershik map can be defined, and

\vspace{2mm}

\noindent \textup{(B)}
 $\textup{int}E_{0,\infty,\min} = \textup{int}E_{0,\infty,\max} = \kuu$
 or
 $\abs{\textup{int}E_{0,\infty,\min}}
 = \abs{\textup{int}E_{0,\infty,\max}} =1$.
\end{prop}
\begin{proof}
Suppose that $(V,E,\ge)$ is decisive.
Then, the Vershik map can be extended as a homeomorphism.
Therefore, Condition (A) is evident.
If exactly one of $\textup{int}E_{0,\infty,\min} \nekuu$
 and $\textup{int}E_{0,\infty,\max} \nekuu$ holds,
then it is evident that any homeomorphic Vershik map is not possible.
If both $\textup{int}E_{0,\infty,\min} \nekuu$
 and $\textup{int}E_{0,\infty,\max} \nekuu$ hold,
then the uniqueness of Vershik map implies 
 $\abs{\textup{int}E_{0,\infty,\min}}
 = \abs{\textup{int}E_{0,\infty,\max}} =1$.
Therefore, we can conclude Condition (B).
Suppose that both (A) and (B) hold.
Then, the existence of the homeomorphic Vershik map is obtained from 
 Condition (A).
The uniqueness follows from Condition (B).
\end{proof}
In \cite{DownarowiczKarpel_2019DecisiveBratteliVershikmodels},
 Downarowicz and Karpel presented the condition for a zero-dimensional system
 to be Bratteli--Vershikizable.
Based on \cref{thm:bijectivetriplesBVmodels},
 we show a new proof using \cref{prop:minimalqscontdic}.
\begin{thm}[Downarowicz and Karpel Theorem 3.1\cite{DownarowiczKarpel_2019DecisiveBratteliVershikmodels}]\label{thm:Bratteli-Vershikizability-new}
A zero-dimensional system $(X,f)$
is Bratteli--Vershikizable if and only if either the set of aperiodic points is
dense, or its closure misses one periodic orbit.
\end{thm}
\begin{proof}
The `only if' part proceeds as in 
 \cite{DownarowiczKarpel_2019DecisiveBratteliVershikmodels}.
To show this,
 let $(X,f)$ be a zero-dimensional system that is Bratteli--Vershikizable.
Let $(V,E,\ge,\psi)$ be a decisive Bratteli--Vershik model that is
 topologically conjugate to $(X,f)$.
Then, $\text{int}E_{0,\infty,\min}$ is at most a single point $x_0$;
 otherwise, the Vershik map is not unique.
We remove the orbit $O(x_0)$ if necessary and we assume that
 $\text{int}E_{0,\infty,\min} \iskuu$.
Suppose that the set of aperiodic points is not dense.
Then, the set of periodic points contains an open set.
By the Baire category theorem, there exists a positive integer $n$
 such that the set of periodic points with the least period $n$
 contains a non-empty open set $U$.
This is impossible because $\text{int}E_{0,\infty,\min} \iskuu$.
We therefore conclude that the set of aperiodic points is dense.
If $O(x_0)$ is a periodic orbit, then it is an isolated periodic
 orbit, as desired.
If $O(x_0)$ is aperiodic, then after recovering $O(x_0)$ the system
 is still densely aperiodic, as desired.
This concludes the `only if' part.
To show the `if' part, let $(X,f)$ be a zero-dimensional system
 such that the set of aperiodic points is
dense, or its closure misses one periodic orbit.
We remove the isolated periodic orbit if one exists.
Then, $(X,f)$ is a densely aperiodic zero-dimensional system.
Let $B$ be a minimal quasi-section.
By \cref{prop:minimalqscontdic},
 we get that $\textrm{int}B \iskuu$.
Thus, the triple $(X,f,B)$ is continuously decisive.
By \cref{thm:bijectivetriplesBVmodels}, there exists
 a Bratteli--Vershik model constructed from $(X,f,B)$.
Finally, owing to \cref{prop:basic-concept-from-DK} that is an observation
 by Downarowicz and Karpel, we obtain that
 every Bratteli--Vershik model constructed from $(X,f,B)$ is decisive and
 also continuously decisive.
We only need to recover the isolated periodic orbit if one existed.
\end{proof}

\begin{prop}\label{prop:from-a-decisive-B-diagram-to-a-class-of-zero-d-sys-triple}
Let $(V^1,E^1,\ge)$ and $(V^2,E^2,\ge)$ be equivalent
 ordered Bratteli diagrams.
Suppose that $(V^1,E^1,\ge)$ is decisive.
Then, the other is also decisive and
 we obtain unique zero-dimensional systems $({E^1}_{0,\infty},\psi_1)$
 and $({E^2}_{0,\infty},\psi_2)$ such that the triples
 $({E^1}_{0,\infty},\psi_1,{E^1}_{0,\infty,\min})$
 and $({E^2}_{0,\infty},\psi_2,{E^2}_{0,\infty,\min})$
 are topologically conjugate.
\end{prop}
\begin{proof}
Let $(V^1,E^1,\ge)$ and $(V^2,E^2,\ge)$ be equivalent
 ordered Bratteli diagrams.
Suppose that $(V^1,E^1,\ge)$ is decisive.
Let $(V,E,\ge)$ be an ordered Bratteli diagram such that
 $(V^1,E^1,\ge)$ and $(V,E,\ge)$ have a common telescoping, and $(V,E,\ge)$ and  $(V^2,E^2,\ge)$ have a common telescoping. Because $(V^1,E^1,\ge)$ is decisive, there exists a unique Bratteli--Vershik
 model $(E^1_{0,\infty},\psi_1)$.
It is evident that we can obtain an isomorphism
 $(E^1_{0,\infty},\psi_1,E^1_{0,\infty,\min})
 = (E_{0,\infty},\psi,E_{0,\infty,\min})$.
Particularly, we find that $(E,V,\ge)$ is also decisive.
In the same way, we find an isomorphism 
 $(E_{0,\infty},\psi,E_{0,\infty,\min})
 = (E^2_{0,\infty},\psi_2,E^2_{0,\infty,\min})$.
In particular, we find that $(E^2,V^2,\ge)$ is also decisive.
We have already obtained an isomorphism 
 $(E^1_{0,\infty},\psi_1,E^1_{0,\infty,\min})
 = (E^2_{0,\infty},\psi_2,E^2_{0,\infty,\min})$.
This concludes the proof.
\end{proof}

\vspace{3mm}

\noindent \textit{Proof of \cref{thm:bijective-continuously-decisive}}

Let $(V,E,\ge)$ be a continuously decisive ordered Bratteli diagram.
Because it is decisive, we can obtain a unique 
 Bratteli--Vershik model $(E_{0,\infty},\psi)$.
Thus, we can obtain a triple $(E_{0,\infty},\psi,E_{0,\infty,\min})$.
By \cref{prop:from-a-decisive-B-diagram-to-a-class-of-zero-d-sys-triple},
 we can also conclude that equivalent ordered Bratteli diagrams
 bring about the same topological conjugacy
 class of triples.
It is evident that $E_{0,\infty,\min}$ is a quasi-section.
We find that $E_{0,\infty,\min}$ has an empty interior as per the definition
 of continuous decisiveness.
Conversely, let $(X,f,B)$ be a continuously decisive triple
 of a zero-dimensional
 system $(X,f)$ with a quasi-section $B$.
Then, by \cref{thm:bijectivetriplesBVmodels}, we obtain
 a Bratteli--Vershik model $(V,E,\ge,\psi)$ such that $E_{0,\infty,\min} = B$
 by a topological conjugacy.
By \cref{thm:triple-implies-equivalent-ordered-Bratteli-diagrams},
 the ordered Bratteli diagram $(V,E,\ge)$ is unique up to equivalence.
Finally, by \cref{prop:basic-concept-from-DK}, these Bratteli--Vershik models
 are decisive and also continuously decisive.
\qed

From the proof of \cref{thm:bijective-continuously-decisive},
 the \cref{thm:bijectivebasicsetsBVmodelswithclosingproperty,thm:bijectiveqsbasicsetsBVmodelswithssclosingproperty} can be transferred to the level of
 ordered Bratteli diagrams.
We list the following obvious consequences:

\begin{cor}\label{cor:bijective-continuously-decisive-closing}
There exists a bijective correspondence between the equivalence classes of
 continuously decisive ordered Bratteli diagrams with closing property
 and the topological conjugacy
 classes of continuously decisive triples of
 zero-dimensional systems with basic sets.
\end{cor}
\begin{proof}
By \cref{thm:closing-basic} and
 the proof of \cref{thm:bijective-continuously-decisive},
 a proof is self-evident.
\end{proof}

\begin{cor}\label{cor:bijective-continuously-decisive-quasi-simple}
There exists a bijective correspondence between the equivalence classes of
 continuously decisive quasi-simple ordered Bratteli diagrams
 and the topological conjugacy
 classes of continuously decisive triples of
 zero-dimensional systems with quasi-simple quasi-sections.
\end{cor}
\begin{proof}
By \cref{thm:quasi-simple}
 and the proof of \cref{thm:bijective-continuously-decisive},
 a proof is self-evident.
\end{proof}

\section{Applications}\label{sec:applic}

Firstly, we confirm that
 the Bratteli--Vershikizable systems are realized by
 the Bratteli--Vershik models with the stronger conditions
 that are described above in this paper.

\begin{thm}\label{thm:quasi-simple-decisive-and-closing}
Let $(X,f)$ be a Bratteli--Vershikizable zero-dimensional system.
Then, there exists a quasi-simple decisive Bratteli--Vershik model
 $(V,E,\ge,\psi)$ of $(X,f)$ with the closing property.
\end{thm}
\begin{proof}
Let $(X,f)$ be a Bratteli--Vershikizable zero-dimensional system.
Owing to
 \cref{thm:Bratteli-Vershikizability-new} or directly from
 \cite[Theorem 3.1]{DownarowiczKarpel_2019DecisiveBratteliVershikmodels},
 either the set of aperiodic points is
dense, or its closure misses one periodic orbit.
If there exists a unique isolated periodic orbit,
 then we exclude it from $(X,f)$; getting densely aperiodic $(X,f)$.
Therefore, by \cref{thm:qs-main},
 there always exists
 a quasi-simple continuously decisive basic $B \subseteq X$.
Then,
 by \cref{cor:bijective-continuously-decisive-closing,cor:bijective-continuously-decisive-quasi-simple},
 we obtain a quasi-simple continuously decisive
 Bratteli--Vershik model $(V,E,\ge,\psi)$ with closing property
for $(X,f)$
 with $E_{0,\infty,\min} = B$ by topological conjugacy.
By recovering the isolated periodic orbit if necessary, we obtain
 a quasi-simple decisive ordered Bratteli diagram with closing property.
\end{proof}

We show that in a large class of zero-dimensional systems,
 closing property implies decisiveness.

\begin{thm}\label{thm:da-non-wandering-closing-decisive}
Let $(X,f)$ be densely aperiodic and $\Omega(f) = X$.
Suppose that $(V,E,\ge,\psi)$ is a corresponding Bratteli--Vershik model with
 closing property.
Then, $(V,E,\ge)$ is decisive.
\end{thm}
\begin{proof}
By the assumption and by \cref{thm:closing-basic},
 it follows that $E_{0,\infty,\min}$ is a basic set.
Because $(X,f)$ is densely aperiodic,
 it follows that $E_{0,\infty,\min}$ is continuously decisive
 by \cref{thm:ap-non-wandering}.
By \cref{prop:basic-concept-from-DK}, we conclude that
 $(V,E,\ge)$ is decisive.
\end{proof}

\begin{cor}\label{cor:tt-closing-decisive}
Let $(X,f)$ be a topologically transitive zero-dimensional system.
Suppose that $(V,E,\ge,\psi)$ is a corresponding Bratteli--Vershik model with
 closing property.
Then, $(V,E,\ge)$ is decisive.
\end{cor}
\begin{proof}
Let $(X,f)$ be a topologically transitive zero-dimensional system.
Then, by \cite[Proposition 1.1]{Silverman_1992OnMapsWithDenseOrbitsAndTheDefinitionOfChaos}, there exists a dense orbit.
Suppose that the dense orbit is a periodic orbit.
Then, we obtain that $(V,E,\ge,\psi)$ is decisive.
Suppose that the dense orbit is not a periodic orbit.
Then, $(X,f)$ is densely aperiodic.
Because $(X,f)$ is topologically transitive,
 we obtain  $\Omega(f) = X$.
Therefore, the conclusion follows
 from \cref{thm:da-non-wandering-closing-decisive}.
\end{proof}

Finally, we could not identify Bratteli--Vershik models
 that are topologically conjugate to $(X,f,\inf_f)$.

\vspace{5mm}

\noindent
\textsc{Acknowledgments:}
This work was partially supported by JSPS KAKENHI (Grant Number JP20K03643).
I would like to thank Editage (www.editage.jp)
 for providing English-language editing services.

\bibliographystyle{amsalpha}
\providecommand{\bysame}{\leavevmode\hbox to3em{\hrulefill}\thinspace}
\providecommand{\MR}{\relax\ifhmode\unskip\space\fi MR }
\providecommand{\MRhref}[2]{%
  \href{http://www.ams.org/mathscinet-getitem?mr=#1}{#2}
}
\providecommand{\href}[2]{#2}

\end{document}